\documentclass[a4paper]{amsart}
\usepackage{hyperref}
\usepackage{txfonts, amsmath,amstext,amsthm,amscd,amsopn,verbatim,amssymb, amsfonts}
\usepackage{fullpage}

\usepackage[bbgreekl]{mathbbol}
\usepackage{enumerate}
\usepackage{tikz}
\usetikzlibrary{matrix}
\usetikzlibrary{shapes}
\usetikzlibrary{arrows}
\usetikzlibrary{calc,3d}
\usetikzlibrary{decorations,decorations.pathmorphing,decorations.pathreplacing,decorations.markings}
\usetikzlibrary{through}

\tikzset{snakeit/.style={decorate, decoration={snake, amplitude=.2mm,segment length=1mm}}}

\tikzset{ext/.style={circle, draw,inner sep=1pt},int/.style={circle,draw,fill,inner sep=1.4pt},nil/.style={inner sep=1pt}}
\tikzset{cy/.style={circle,draw,fill,inner sep=2pt},scy/.style={circle,draw,inner sep=2pt},scyx/.style={draw,cross out,inner sep=2pt},scyt/.style={draw,regular polygon,regular polygon sides=3,inner sep=0.95pt}}
\tikzset{exte/.style={circle, draw,inner sep=3pt},inte/.style={circle,draw,fill,inner sep=3pt}}
\tikzset{diagram/.style={matrix of math nodes, row sep=3em, column sep=2.5em, text height=1.5ex, text depth=0.25ex}}
\tikzset{diagram2/.style={matrix of math nodes, row sep=0.5em, column sep=0.5em, text height=1.5ex, text depth=0.25ex}}

\theoremstyle{plain}
  \newtheorem{thm}{Theorem}
  
  \newtheorem{prop}{Proposition}

  \newtheorem{conjecture}{Conjecture}
  \newtheorem{lemma}{Lemma}
\theoremstyle{definition}

\newcommand{\Hom}{\mathrm{Hom}}

\newcommand{\K}{{\mathbb{K}}}

\newcommand{\Graphs}{{\mathsf{Graphs}}}
\newcommand{\fGraphs}{{\mathsf{fGraphs}}}

\newcommand{\Gra}{{\mathsf{Gra}}}

\newcommand{\Tw}{\mathit{Tw}}
\newcommand{\Def}{\mathrm{Def}}

\newcommand{\op}{\mathcal}

\newcommand{\Lie}{\mathsf{Lie}}

\newcommand{\hoLie}{\mathsf{hoLie}}

\newcommand{\fGC}{\mathrm{fGC}}

\newcommand{\Ass}{\mathsf{Assoc}}
\newcommand{\Com}{\mathsf{Com}}

\newcommand{\bpm}{\begin{pmatrix}}
\newcommand{\epm}{\end{pmatrix}}

\newcommand{\GC}{\mathrm{GC}}
\newcommand{\HGC}{\mathrm{HGC}}

\newcommand{\bbS}{\mathbb{S}}

\newcommand{\gra}{\mathit{gra}}

\newcommand{\Poiss}{\mathsf{Poiss}}
\newcommand{\hoPoiss}{\mathsf{hoPoiss}}

\newcommand{\tp}{{\circlearrowleft}}

\begin{document}
\title{Differentials on graph complexes II - hairy graphs}

\author{Anton Khoroshkin}
\address{1004 Faculty of mathematics HSE, 7 Vavilova street, Moscow, Russia, 115280}
\email{akhoroshkin@hse.ru}

\author{Thomas Willwacher}
\address{Institute of Mathematics\\ University of Zurich\\ Winterthurerstrasse 190 \\ 8057 Zurich, Switzerland}
\email{thomas.willwacher@math.uzh.ch}

\author{Marko \v Zivkovi\' c}
\address{Institute of Mathematics\\ University of Zurich\\ Winterthurerstrasse 190 \\ 8057 Zurich, Switzerland}
\email{the.zivac@gmail.com}

\thanks{
A.K. has been partially supported by RFBR grants 13-02-00478, 13-01-12401, 
by "The National Research University--Higher School of Economics" Academic Fund Program in 2013-2014,
research grant 14-01-0124, by Dynasty foundation and Simons-IUM fellowship.
T.W. and M.\v Z. have been partially supported by the Swiss National Science foundation, grant 200021\_150012.
All three authors have been supported by the SwissMAP NCCR funded by the Swiss National Science foundation.
}



\begin{abstract}
We study the cohomology of the hairy graph complexes which compute the rational homotopy of embedding spaces, generalizing the Vassiliev invariants of knot theory.
We provide spectral sequences converging to zero whose first pages contain the hairy graph cohomology.
Our results yield a way to construct many hairy graph cohomology classes out of non-hairy classes by a mechanism which we call the waterfall mechanism. By this mechanism we can construct many previously unknown classes and provide a first glimpse at the tentative global structure of the hairy graph cohomology.
\end{abstract}

\maketitle

\section{Introduction}
The graph complexes in its various flavors are some of the most intriguing objects of homological algebra.
Determining their cohomology is arguably one of the most fundamental open problem in the field, and there are few mathematical tools available to obtain information on this cohomology.

In this paper we study one type of graph complexes, namely the complexes of linear combinations of ordinary graphs with external legs (``hairs''), such as the following examples:
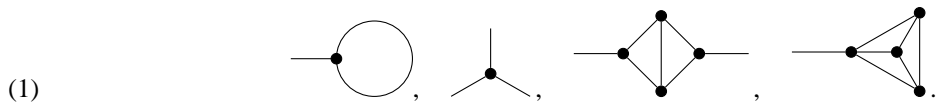
\begin{equation}\label{equ:hairysamples}
 \begin{tikzpicture}[scale=.5]
 \draw (0,0) circle (1);
 \draw (-180:1) node[int]{} -- +(-1.2,0);
 \end{tikzpicture}
,\quad
\begin{tikzpicture}[scale=.6]
\node[int] (v) at (0,0){};
\draw (v) -- +(90:1) (v) -- ++(210:1) (v) -- ++(-30:1);
\end{tikzpicture}
\,,\quad
\begin{tikzpicture}[scale=.5]
\node[int] (v1) at (-1,0){};\node[int] (v2) at (0,1){};\node[int] (v3) at (1,0){};\node[int] (v4) at (0,-1){};
\draw (v1)  edge (v2) edge (v4) -- +(-1.3,0) (v2) edge (v4) (v3) edge (v2) edge (v4) -- +(1.3,0);
\end{tikzpicture}
 \, ,\quad
 \begin{tikzpicture}[scale=.6]
\node[int] (v1) at (0,0){};\node[int] (v2) at (180:1){};\node[int] (v3) at (60:1){};\node[int] (v4) at (-60:1){};
\draw (v1) edge (v2) edge (v3) edge (v4) (v2)edge (v3) edge (v4)  -- +(180:1.3) (v3)edge (v4);
\end{tikzpicture}
 \, .
\end{equation}
Due to different possible choices regarding the gradings and signs associated to symmetries, the complexes of such hairy graphs in fact come in various flavors, indexed by a pair of integers $(m,n)$. We denote the corresponding hairy graph complexes by $\HGC_{m,n}$, a detailed definition can be found in section \ref{sec:hairygraphs} below.
These complexes compute the rational homotopy of the spaces of embeddings of disks modulo immersions, fixed at the boundary
\[
 \overline{\mathrm{Emb}_\partial}(\mathbb{D}^m,\mathbb{D}^n)
\]
provided that $n-m\geq 3$, cf. \cite{FTW, AT, Tur1}. Furthermore, the diagrams enumerating Vassiliev invariants of knot theory appear as the top cohomology of the hairy graph complex $\HGC_{1,3}$.
The long standing open problem we are attacking in this paper is the following:

\smallskip

{\bf Open Problem:} Compute the hairy graph cohomology $H(\HGC_{m,n})$.

\smallskip

Unfortunately, we currently have very few tools available for computing this cohomology, or even for guiding our intuition as to what the correct answer might be. The purpose of this paper is to introduce a new tool to attack the above open problem. It will not give us a complete answer as to what $H(\HGC_{m,n})$ is, but it will allow us to construct (infinitely) many new hairy graph cohomology classes, and display a rich set of constraints among the hairy classes, thus giving a glimpse of the global structure of $H(\HGC_{m,n})$.

Before describing our results, let us review some known basic facts and previous results about the hairy graph cohomology.
First, from the definition of the complexes $\HGC_{m,n}$ given below it will be evident that they split into subcomplexes according to the number of hairs and the Betti number (i.e., the loop order) of graphs.
In other words, the complexes $\HGC_{m,n}$ are tri-graded, by the cohomological degree, the number of hairs, and the loop order.
The subcomplexes of fixed numbers of hairs and loops are finite dimensional.
Furthermore, the complexes $\HGC_{m,n}$ and $\HGC_{m',n'}$ are isomorphic up to some unimportant degree shifts if $m\equiv m' \text{ mod } 2$ and 
$n\equiv n' \text{ mod } 2$. Hence it suffices to understand 4 possible cases according to parity of $m$ and $n$.

Several less trivial partial results have been obtained about the hairy graph cohomology in the last decades.
Classes in low degrees have been computed by hand or with computer assistance \cite{kneissler}.
Furthermore, it is known \cite[Propositions 4.1 and 4.4]{vogel} that the 2-hair-subspaces $H^{-1}(\HGC_{1,3}^2)$ and $H^{1}(\HGC_{2,3}^2)$ are each isomorphic to the non-hairy graph cohomology $H^{-3}(\GC_{3})$.

Now let us describe the results of this paper, and the proposed line of attack on the open problem above.
We construct deformations of the graph differentials to differentials for which the cohomology can be computed for $m$ even, and can conjecturally be computed for $m$ odd. The spectral sequences thus obtained can be used to obtain information about the non-deformed cohomology.
Concretely, our main result will be the following.

\begin{thm}\label{thm:main}
For each $n$ there is a differential $D$ on $\HGC_{0,n}$ with the following properties:
\begin{enumerate}
\item $D$ deforms the standard differential $\delta$: The operation $D-\delta$ decreases the number of hairs by at least one. 
\item $D$ preserves the grading on $\HGC_{0,n}$ by the number 
\[
(\text{loop order})+(\text{number of hairs})
\]
\item We have $H(\HGC_{0,n},D)=0$.
\item The spectral sequence associated to the filtration by number of hairs converges to $H(\HGC_{0,n},D)=0$, and its first page consists of the hairy graph cohomology $H(\HGC_{0,n})=H(\HGC_{0,n}, \delta)$.
\end{enumerate}
\end{thm}

For $m$ odd we can also construct a similar deformation of the differential and hence an associated spectral sequence. We conjecture (see Conjecture \ref{conj:main} below) that the resulting deformed complex is acyclic.
This conjecture is consistent with computer data for the hairy graph cohomology, which is available up to loop order 5.\footnote{In fact, the third author has a proof of Conjecture \ref{conj:main} for $n$ even, which will appear separately as a part of his thesis \cite{MarkoThesis}.}

\smallskip

Next, let us describe the implications of Theorem \ref{thm:main} for the hairy graph cohomology.
To this end, we need to recall one more ingredient: By results of V. Turchin and the second author \cite{grt,TW} it is known that on each of the complexes $\HGC_{n,n}$ and $\HGC_{n-1,n}$ there is a deformation of the differential (say $D'$) such that the cohomology of the deformed complex is equal to the ordinary (non-hairy) graph cohomology:
\begin{align}\label{equ:HGCGC}
 H(\HGC_{n,n},D')&\cong H(\GC_n) & H(\HGC_{n-1,n},D')&\cong H(\GC_n).
\end{align}
These results and the definition of $D'$ in each case will be recalled in more detail in section \ref{sec:TWrecollection} below.
Here $\GC_n$ is the non-hairy graph complex, defined similarly to $\HGC_{m,n}$, except that graphs are not allowed to have hairs. The result \eqref{equ:HGCGC} is interesting because of the following facts:
\begin{itemize}
 \item The structure of $H(\GC_n)$ is significantly better understood then the structure of its hairy counterpart $H(\HGC_{m,n})$. In particular, one knows large families of cohomology classes explaining all cohomology in the computer accessible regime, and one has certain vanishing conjectures, see \cite{KWZ} for an overview.
 \item From \eqref{equ:HGCGC} and a spectral sequence argument one can in particular see that the non-hairy graph cohomology $H(\GC_n)$ embeds into $H(\HGC_{m,n})[1]$ for all $m,n$. Concretely, given a non-hairy graph cocycle $\gamma\in \GC_n$ the corresponding hairy graph cocycle is obtained by summing over all ways of attaching one hair to $\gamma$, pictorially:
 \[
  \GC_n \ni \gamma \mapsto \sum  
  \begin{tikzpicture}[baseline=-.65ex]
   \node (v) at (0,0.1) {$\gamma$};
   \draw (v) edge +(0,-.5);
  \end{tikzpicture}
  \in \HGC_{m,n}[1].
 \]
In particular, note that the hairy graph cohomology classes thus obtained all live in the one-hair piece of the hairy graph cohomology.
\item The differentials $D'$ respect the grading on the hairy graph complex by loop order.
\end{itemize}

\medskip

{\bf Constraints on the cohomology and the waterfall mechanism.} Let us now describe how to construct from the above two spectral sequences a large set of additional non-trivial hairy graph cohomology classes by a process we call the waterfall mechanism. We call the spectral sequence arising from the deformed differential $D$ of Theorem \ref{thm:main}the \emph{first} spectral sequence, and the one arising from $D'$ the \emph{second}.
%
%
%
Let us focus on the case of $m$, $n$ even for concreteness, say $m=n=0$. The case of $m$ even and $n$ odd (and the cases $m$ odd, $n$ even or odd provided Conjecture \ref{conj:main}) is treated analogously.
%
%

Note that the convergence of the spectral sequence of Theorem \ref{thm:main} implies that the hairy graph cohomology classes must come ``in pairs''. More concretely, given a hairy graph cohomology class $\Gamma$, it will survive up to some page of the spectral sequence, on which it is either killed by or kills (the image of) another hairy graph cohomology class.
More concretely, from (2) of Theorem \ref{thm:main} we see that if $\Gamma$ lives in tri-degree
\[
 (\text{cohom. degree}, \text{number of hairs}, \text{loop order})=:(d,h,l),
\]
then the ``partner class'' that it kills (or is killed by) must live in tri-degree $(d+1, h-j, l+j)$ (or $(d-1, h+j, l-j)$) for some yet unknown positive integer $j$.
Hence from the existence of the non-trivial class $\Gamma\in\HGC_{0,0}$ we can conclude that there is another nontrivial class in $\HGC_{0,0}$ whose tri-degree lies on a union of half-lines in $\mathbb{Z}^3$. 
A representative of this (or rather, some such) class may be constructed by following the spectral sequence.

Now consider the second spectral sequence arising from the deformed differential $D'$ of \eqref{equ:HGCGC}.
As before, nontrivial hairy graph cohomology classes which are not in the image of $H(\GC_0)$ must kill or be killed by other non-trivial classes on some page of the spectral sequence. For this second spectral sequence, one can see that the partner class of a class in tri-degree $(d,h,l)$ must live in tri-degree $(d+1,h+j,l)$ (or $(d-1,h-j,l)$) for a positive integer $j$. (In fact, it will be shown in the upcoming work \cite{TW2} that the spectral sequence abuts on the second page and hence $j=1$.)

Now, using the constraints provided by the first and the second spectral sequences together, we may construct a large set of hairy graph cohomology classes from (assumed to be known) non-hairy classes.
Concretely, consider a non-hairy class $\gamma\in H(\GC_0)$. As above, by adding one hair we obtain a non-trivial hairy class $\Gamma$. It must be killed by (the image of) some other class, say $\Gamma_1$, in the first spectral sequence (the one from Theorem \ref{thm:main}). The class $\Gamma_1$ must necessarily have more than one hair. Hence it must kill or be killed by (the image of) some other class in the second spectral sequence. This class must again be killed by (the image of) some class in the first spectral sequence etc., until at some point we reach another hairy graph cohomology class in the image of $H(\GC_0)$. 
By this process we conclude from the existence of a non-hairy graph cohomology class the existence of a string 
of hairy graph cohomology classes. For an illustration of the process, see the computer generated table of the hairy graph cohomology in Figure \ref{fig:cancellatione}, in which (some of) the cancellations in the two spectral sequences have been inscribed. We call the above mechanism to construct strings of hairy classes from non-hairy the ``waterfall mechanism'', by visual similarity of the cancellation pattern to a waterfall.

This paper may be seen as a continuation of \cite{KWZ}, where we used similar methods to study the ordinary (non-hairy) graph complexes.

\subsection*{Structure of the paper}
In section \ref{sec:background} we recall the relevant definitions.
Section \ref{sec:spectrale} is dedicated to the construction of the deformed differentials and the spectral sequences for $m$ even, and in particular the proof of Theorem \ref{thm:main}. In section \ref{sec:waterfalle} we furthermore investigate how these spectral sequences can be used to construct many nontrivial hairy graph cohomology classes by the ``waterfall mechanism''.

The analogous construction of the deformed differentials for odd $m$ is carried out in section \ref{sec:spectralo}. We leave the vanishing result analogous to Theorem \ref{thm:main} open for odd $m$, see Conjecture \ref{conj:main}.

\bibliographystyle{plain}

\section{Background and definitions}\label{sec:background}
In this section we will recall basic notation and several results shown in the literature that will be used below, for the readers convenience.
\subsection{Basic Notation}
We will work over a ground field $\K$ of characteristic zero. All vector spaces and differential graded vector spaces are assumed to be $\K$-vector spaces. The phrase differential graded will be abbreviated by dg.
We use cohomological conventions, so that the degree of the differentials is $+1$.
We denote the subspace of elements of homogeneous degree $k$ of a graded vector space $V$ by $V^k$. We define the degree shifted vector space $V[r]$ such that $(V[r])^k\cong V^{k+r}$.

We will use the language of operads. A good introduction is found in the standard textbook \cite{lodayval}.
The associative, commutative and Lie operads are denoted by $\Ass,\Com,\Lie$ respectively.
We denote by $\Poiss_n$ the $n$-Poisson operad generated by a binary commutative product $-\wedge -$ of degree zero, and a compatible Lie bracket $[-,-]$ of degree $1-n$.

The $r$-fold operadic (de)suspension is denoted by $\op P\{r\}$. It is defined such that if the vector space $V$ carries a $\op P\{r\}$-algebra structure then $V[r]$ carries a $\op P$-algebra structure.
We denote by $\Omega(\op C)$ the cobar construction of coaugmented cooperad, and by $\op P^\vee$ the Koszul dual cooperad to $\op P$. The canonical minimal resolutions of the "standard" operads are denoted by $\hoLie=\Omega(\Lie^\vee)$, $\hoPoiss_n=\Omega(\Poiss_n^\vee)$, etc. We will furthermore abbreviate $\hoLie_n=\hoLie\{n-1\}$ so that we have a natural operad map $\hoLie_n\to \hoPoiss_n$.

Suppose we are given some operad map $f: \Omega(\op C)\to \op P$. Such a map describes a Maurer-Cartan element (say $\alpha_f$) in the operadic convolution dg Lie algebra
\[
 \mathrm{Conv}(\op C,\op P)= \prod_N \Hom_{\bbS}(\op C(N), \op P(N)),
\]
see \cite[section 6.4.4]{lodayval} for the definition.
We define the deformation complex of the operad map $f$ to be the convolution dg Lie algebra, twisted by the Maurer-Cartan element $\alpha_f$ corresponding to $f$,
\[
 \Def(\Omega(\op C)\xrightarrow{f} \op P) := \mathrm{Conv}(\op C,\op P)^{\alpha_f}.
\]

\subsection{M. Kontsevich's graph complexes}
We quickly recall the construction of the (commutative) graph complexes. For more details see \cite{grt}.
Consider the set of connected directed graphs $\gra_{N,k}$ with $N$ vertices (uniquely) labelled by numbers $\{1,\dots,N\}$ and $k$ edges labelled set $\{1,\dots,k\}$.
There is a natural right action of the group $S_N\times (S_k\ltimes S_2^{k})$ by permuting the labels and changing the direction of edges.
We define the operad $\Gra_d$ such that 
\[
\Gra_d(N) = \oplus_k \left( \K \langle \gra_{N,k} \rangle \otimes \K[d-1]^{\otimes k} \right)_{S_k\ltimes S_2^k}
\]
where we declare that $S_k$ acts diagonally, and on the $\K[d-1]$ factors by permutation with sign (if $d$ is even), and $S_2^k$ acts with a sign if $d$ is odd.
The signs are chosen such that there is a map of operads $e_d\to \Gra_d$.
The full graph complex is the deformation complex
\[
\fGC_d =\Def(\hoLie_d\to \Gra_d) \cong  \prod_{N\geq 1} (\Gra_d(N) \otimes \K[-d]^\otimes N)^{S_N}[d].
\]
It carries a natural dg Lie algebra structure.
We will in particular consider two subcomplexes
\[
\GC_d\subset \GC_d^2\subset \fGC_d
\]
where $\GC_d^2$ is spanned by the connected graphs with at least bivalent vertices and $\GC_d$ is spanned by connected graphs with at least trivalent vertices.
We denote the differential on the graph complex by $\delta$. Combinatorially, $\delta$ acts like
\begin{equation}
 \delta\Gamma\,=\sum_{x\in V(\Gamma)}\frac{1}{2}s_x - a_x,
\end{equation}
where $V(\Gamma)$ is the set of vertices of $\Gamma$, $s_x$ stands for ``splitting of $x$'' and means inserting
\begin{tikzpicture}[scale=.5]
 \node[int] (a) at (0,0) {};
 \node[int] (b) at (1,0) {};
 \draw (a) edge[->] (b);
\end{tikzpicture}
instead of the vertex $x$ and summing over all possible ways of connecting the edges that have been connected to $x$ to the new two vertices, and $a_x$ stands for ``Adding an edge at $x$'' and  means adding
\begin{tikzpicture}[scale=.5]
 \node[int] (a) at (0,0) {};
 \node[int] (b) at (1,0) {};
 \draw (a) edge[->] (b);
 \node[above left] at (a) {$\scriptstyle x$};
\end{tikzpicture}
on the vertex $x$. Unless $x$ is an isolated vertex, $a_x$ will cancel one term of the splitting $s_x$.

The graph complex $\GC_d$ splits into a direct product of finite dimensional sub complexes according to loop order.

\subsection{The spectral sequence of \cite{KWZ}}\label{sec:KWZrecollection}
In the first paper of this series we introduced deformed differentials on the graph complexes above. 
For $d=0$ we can deform the differential to $\delta+\nabla$, where the additional operator $\nabla$ is defined as the Lie bracket with the tadpole graph
\[
\nabla =[
\begin{tikzpicture}[baseline=-.65ex]
\node[int](v) at (0,0){};
\draw (v) edge [loop] (v);
\end{tikzpicture}
, \cdot].
\]
Combinatorially, $\nabla$ acts by adding one edge, in all possible ways.

For $d=1$ we may deform the differential in a different way. First, note that there is a map of operads
\[
\Ass\to \Gra_1
\]
by sending the generator to the series of graphs
\[
\sum_{k\geq 0 }\frac 1 {k!}
\begin{tikzpicture}[baseline=-.65ex]
\node[ext](v) at (0,0){1};
\node[ext](w) at (1,0){2};
\draw (v) edge [bend right] (w) edge [bend left] node[above] {$\scriptstyle k\times $} (w) edge (w);
\end{tikzpicture}
\]
We hence obtain a map $\Lie\to \Ass\to \Gra_d$ by composition. We may form the deformation complex
\[
\Def(\hoLie_1\to \Gra_d) =:\fGC^\Theta.
\]
As a graded vector space the right hand side is isomorphic to $\fGC$ but the differential is deformed to, say, 
\[
\delta_\Theta = \delta + 
\sum_{k\geq 1 }\frac 1 {(2k+1)!}
\left[
\begin{tikzpicture}[baseline=-.65ex]
\node[int](v) at (0,0){};
\node[int](w) at (1,0){};
\draw (v) edge [bend right] (w) edge [bend left] node[above] {$\scriptstyle 2k+1\times $} (w) edge (w);
\end{tikzpicture}
,\,\cdot \,\right]
\]

The main results of \cite {KWZ} is the following.
\begin{thm}[Theorem 2 and Corollary 4 of \cite{KWZ}]
\begin{align}
H(\GC_0^2,\delta+\nabla) &\cong \K[1]
\\
H(\GC_1^2, \delta_\Theta) &\cong \K[3]
\end{align}
\end{thm}

As a consequence one obtains spectral sequences converging to (essentially) 0, whose first page contains $H(\GC_d^2, \delta)$, for $d=0,1$.

\subsection{Graph operads and operadic twisting}\label{sec:graphsop}
Given an operad $\op P$ with a map $\hoLie_d\to \op P$ one may apply the formalism of operadic twisting \cite{vastwisting} to produce another operad $\Tw\op P$ which has the property that $\Tw\op P$ algebras may be "naturally" twisted by Maurer-Cartan elements. 
Furthermore, $\Tw\op P$ comes with a natural action of the deformation complex $\Def(\hoLie_d\to \op P)$.
We will consider the twisted operad
\[
\fGraphs_d = \Tw \Gra_d.
\]
Elements of $\fGraphs_d(N)$ are series of graphs with $N$ numbered ("external") vertices and an arbitrary number of indistinguishable ("internal") vertices, for example the following:
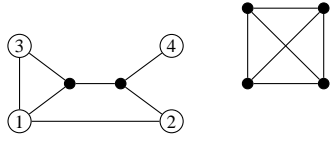
\[
 \begin{tikzpicture}[baseline=2ex]
  \node[ext] (v1) at (0,0) {$\scriptstyle 1$};
  \node[ext] (v2) at (2,0) {$\scriptstyle 2$};
  \node[ext] (v3) at (0,1) {$\scriptstyle 3$};
  \node[ext] (v4) at (2,1) {$\scriptstyle 4$};
  \node[int] (w1) at (.66,.5) {};
  \node[int] (w2) at (1.33,.5) {};
  \node[int] (w3) at (3,.5) {};
  \node[int] (w4) at (4,.5) {};
  \node[int] (w5) at (3,1.5) {};
  \node[int] (w6) at (4,1.5) {};
  \draw (v1) edge (v2) edge (v3) edge (w1)
        (w1) edge (w2) edge (v3)
        (w2) edge (v2) edge (v4)
        (w3) edge (w4) edge (w5) edge (w6)
        (w4) edge (w5) edge (w6)
        (w5) edge (w6);
 \end{tikzpicture}
\]
The operad $\fGraphs_d$ contains a suboperad $\Graphs_d$ (defined by Kontsevich \cite{K2}) spanned by graphs such that each connected component contains at least one external vertex and such that each internal vertex is at least trivalent. For example, the graph shown above does not satisfy these criteria, but the one below does:
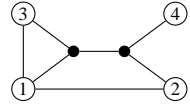
\[
 \begin{tikzpicture}[baseline=2ex]
  \node[ext] (v1) at (0,0) {$\scriptstyle 1$};
  \node[ext] (v2) at (2,0) {$\scriptstyle 2$};
  \node[ext] (v3) at (0,1) {$\scriptstyle 3$};
  \node[ext] (v4) at (2,1) {$\scriptstyle 4$};
  \node[int] (w1) at (.66,.5) {};
  \node[int] (w2) at (1.33,.5) {};
  \draw (v1) edge (v2) edge (v3) edge (w1)
        (w1) edge (w2) edge (v3)
        (w2) edge (v2) edge (v4);
 \end{tikzpicture}
\]
Combinatorially the differential $\delta$ on $\Graphs_d$ is given by summing over all ways of splitting an (either external or internal) vertex, producing one additional internal vertex. Pictorially:
\begin{align}\label{equ:graphsdiff}
\delta
 \begin{tikzpicture}[baseline=-.65ex, scale=.7]
  \node[int] (v) at (0,0) {};
  \draw (v) edge +(-.5,.5) edge +(-.5,0) edge +(-.5,-.5);
  \draw (v) edge +(.5,.5) edge +(.5,0) edge +(.5,-.5);
 \end{tikzpicture}
 &=
 \sum
 \begin{tikzpicture}[baseline=-.65ex, scale=.7]
  \node[int] (v) at (0,0) {};
  \node[int] (w) at (1,0) {};
  \draw (v) edge (w);
  \draw (v) edge +(-.5,.5) edge +(-.5,0) edge +(-.5,-.5);
  \draw (w) edge +(.5,.5) edge +(.5,0) edge +(.5,-.5);
 \end{tikzpicture}
 &\text{or}
 & &
 \delta
  \begin{tikzpicture}[baseline=-.65ex, scale=.7]
  \node[ext] (v) at (0,0) {$\scriptstyle j$};
  \draw (v) edge +(-.5,.5) edge +(-.5,0) edge +(-.5,-.5);
  \draw (v) edge +(.5,.5) edge +(.5,0) edge +(.5,-.5);
 \end{tikzpicture}
&=
\sum
\begin{tikzpicture}[baseline=-.65ex, scale=.7]
  \node[ext] (v) at (0,0) {$\scriptstyle j$};
  \node[int] (w) at (1,0) {};
  \draw (v) edge (w);
  \draw (v) edge +(-.5,.5) edge +(-.5,0) edge +(-.5,-.5);
  \draw (w) edge +(.5,.5) edge +(.5,0) edge +(.5,-.5);
 \end{tikzpicture}
\, .
\end{align}
We refer the reader to \cite{K2,grt} for more details.

The important result for us is the following.
\begin{thm}[Kontsevich \cite{K2}, Lambrechts--Voli\'c \cite{LV}]\label{thm:KLV}
The map of operads
\[
\Poiss_d \to \Graphs_d.
\]
given on generators by the assignment
\begin{align*}
- \wedge - &\mapsto 
\begin{tikzpicture}[baseline=-.65ex]
\node[ext] (v) at (0,0) {1};
\node[ext] (w) at (0.7,0) {2};
\end{tikzpicture}
&
[-,-] &\mapsto 
\begin{tikzpicture}[baseline=-.65ex]
\node[ext] (v) at (0,0) {1};
\node[ext] (w) at (1,0) {2};
\draw (v) edge (w) ;
\end{tikzpicture}
\end{align*}
is a quasi-isomorphism. In particular $H(\Graphs_d(1))\cong\K$.
\end{thm}

Furthermore, we note that (from the operadic twisting procedure) one obtains an action of the graph complex $\GC_d$ on the graphs operad $\Graphs_d$ by operadic derivations. In particular, given any Maurer-Cartan element in $\GC_d$ we may construct a deformation of the differential of $\Graphs_d$. Two cases are important for us:
\begin{itemize}
 \item For $d=0$ we can choose the Maurer-Cartan element 
 $\mu=\begin{tikzpicture}[baseline=-.65ex]
\node[int](v) at (0,0){};
\draw (v) edge [loop] (v);
\end{tikzpicture}$. The deformed differential then has the form $\delta+\mu\cdot$, where $\mu\cdot$ is the action of $\mu$ on $\Graphs_0$. Combinatorially, for a graph $\Gamma\in \Graphs_0(r)$, the element $\mu\cdot\Gamma\in \Graphs_0(r)$ is a sum over all graphs obtained by adding one edge between two distinct vertices (internal or external) in all possible ways.
We denote the operad $\Graphs_n$ with the thus twisted differential by $\Graphs_0^\tp$.
\item For $d=1$ we may similarly consider the Maurer-Cartan element 
\[
 \Theta = 
 \sum_{j\geq 1}
\frac{1}{(2j+1)!} \;
\underbrace{
\begin{tikzpicture}[baseline=-.65ex]
 \node[int] (v) at (-.5,0) {};
 \node[int] (w) at (0.5,0) {};
 \node at (0,0) {$\scriptstyle \cdots$};
 \draw (v) edge[bend left] (w)
       (v) edge[bend right] (w);
\end{tikzpicture}
}_{2j+1\text{ edges}}
\]
and obtain a deformed differential $\delta+\Theta\cdot$ on the operad $\Graphs_1$. 
We denote the operad $\Graphs_1$ with the thus twisted differential by $\Graphs_1^\Theta$.
\end{itemize}

\subsection{Hairy graph complexes}\label{sec:hairygraphs}
We consider the operad map $\hoPoiss_m\stackrel{*}{\to} \Graphs_n$ defined as the composition of the natural maps
\[
 \hoPoiss_m \to \Poiss_m\to \Com \to \Poiss_n \to \Graphs_n.
\]
The hairy graph complexes $\HGC_{m,n}$ are subcomplexes of the deformation complexes
\[
\Def(\hoPoiss_m\stackrel{*}{\to} \Graphs_n)\supset \HGC_{m,n}.
\]
Concretely, these complexes consist of maps that factor through $\hoLie_n$ and having images in connected graphs all of whose external vertices have valence one. Alternatively, we may describe elements of $\HGC_{m,n}$ as graphs without bivalent vertices, the univalent vertices having degree $m$ and the other vertices degree $n$ (and the edges degree $1-n$). 
Examples of such graphs are shown in \eqref{equ:hairysamples} in the introduction.

The differential on the hairy graph complexes is the one inheritex from $\Graphs_n$.
Combinatorially it is given by splitting vertices, pictorially
\[
 \delta 
 \begin{tikzpicture}[baseline=-.65ex, scale=.6]
  \node[int] (v) at (0,0){};
  \draw (v) edge +(-.5,-.5) edge +(0,-.5) edge +(.5,-.5)
            edge +(-.5,.5) edge +(0,.5) edge +(.5,.5);
 \end{tikzpicture}
 =
 \sum
  \begin{tikzpicture}[baseline=-.65ex, scale=.6]
  \node[int] (v) at (0,.3){};
  \node[int] (w) at (0,-.3){};
  \draw (w) edge +(-.5,-.5) edge +(0,-.5) edge +(.5,-.5) edge (v)
        (v) edge +(-.5,.5) edge +(0,.5) edge +(.5,.5);
 \end{tikzpicture},
\]
where one sums over splittings producing only vertices of valence at least 3.

The subcomplex $\HGC_{m,n} \subset \Def(\hoPoiss_m\stackrel{*}{\to} \Graphs_n)$ is in fact closed under the natural Lie bracket on the deformation complex. Combinatorially, the induced Lie bracket of two hairy graphs $\Gamma$ and $\Gamma'$ is obtained by summing over all ways of attaching one hair of $\Gamma$ to a vertex of $\Gamma'$, minus the same with $\Gamma$ and $\Gamma'$ interchanged. Pictorially:
\[
\left[ 
\begin{tikzpicture}[baseline=-.8ex]
\node[draw,circle] (v) at (0,.3) {$\Gamma$};
\draw (v) edge +(-.5,-.7) edge +(0,-.7) edge +(.5,-.7);
\end{tikzpicture}
,
\begin{tikzpicture}[baseline=-.65ex]
\node[draw,circle] (v) at (0,.3) {$\Gamma'$};
\draw (v) edge +(-.5,-.7) edge +(0,-.7) edge +(.5,-.7);
\end{tikzpicture}
\right]
=
\sum
\begin{tikzpicture}[baseline=-.8ex]
\node[draw,circle] (v) at (0,1) {$\Gamma$};
\node[draw,circle] (w) at (.8,.3) {$\Gamma'$};
\draw (v) edge +(-.5,-.7) edge +(0,-.7) edge (w);
\draw (w) edge +(-.5,-.7) edge +(0,-.7) edge +(.5,-.7);
\end{tikzpicture}
\pm 
\sum
\begin{tikzpicture}[baseline=-.8ex]
\node[draw,circle] (v) at (0,1) {$\Gamma'$};
\node[draw,circle] (w) at (.8,.3) {$\Gamma$};
\draw (v) edge +(-.5,-.7) edge +(0,-.7) edge (w);
\draw (w) edge +(-.5,-.7) edge +(0,-.7) edge +(.5,-.7);
\end{tikzpicture}.
\]

\subsection{The spectral sequences of \cite{grt,TW,TW2}}\label{sec:TWrecollection}
One can check that the element 
\[
h_0=
\begin{tikzpicture}[baseline=-.65ex]
\draw (0,0)--(.5,0);
\end{tikzpicture}
\]
is a Maurer-Cartan element of the dg Lie algebra $\HGC_{n,n}$.
Likewise, one can check that the element 
\[
h_1= \sum_{k\geq 1} \frac 1 {(2k+1)!}
\underbrace{
\begin{tikzpicture}[baseline=-.65ex]
\node[int] (v) {};
\draw (v) edge +(-.5,-.5) edge +(0,-.5) edge +(.5,-.5);
\end{tikzpicture}
}_{2k+1\times}
\]
is a Maurer-Cartan element in $\HGC_{n-1,n}$.

\begin{thm}[ \cite{TW}, \cite{TW2}, \cite{grt}]\label{thm:TW}
There are quasi-isomorphisms
\begin{align*}
\K \oplus \GC_{n}^2 &\to (\HGC_{n,n}, \delta+[h_0,\cdot]) \\
\K \oplus \GC_{n}^2 &\to (\HGC_{n-1,n}, \delta+[h_1,\cdot]).
\end{align*}
Furthermore the spectral sequence obtained by the filtration by number of hairs abuts at the $E_2$ page in the first case. 
\end{thm}

\section{The spectral sequence: $m$ even}\label{sec:spectrale}
Our goal in this section is to deform the differential $\delta$ on the hairy graph complex $\HGC_{0,n}$ to a new differential $D$ for which the cohomology is computable, and, in fact, trivial.
The construction goes as follows:
The space of unary operations $\Graphs_n(1)$ in the Kontsevich operad $\Graphs_n$ of section \ref{sec:graphsop} may be identified with the completed symmetric algebra $\hat S(\HGC_{0,n})$ as a graded vector space. Indeed, the identification is realized by deleting the external vertex in a graph in $\Graphs_n(1)$ and interpreting the edges previously incident at the external vertex as hairs, as the following example illustrates.
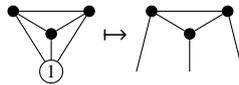
\[
 \begin{tikzpicture}[baseline=-.65ex]
  \node[ext] (v) at (0,-.5) {$\scriptstyle 1$};
  \node[int] (w1) at (0,0) {};
  \node[int] (w2) at (-.5,.3) {};
  \node[int] (w3) at (.5,.3) {};
  \draw (v) edge (w1) edge (w2) edge (w3) 
        (w1) edge (w2) edge (w3) 
        (w2) edge (w3);
 \end{tikzpicture}
\mapsto
 \begin{tikzpicture}[baseline=-.65ex]
  \node[int] (w1) at (0,0) {};
  \node[int] (w2) at (-.5,.3) {};
  \node[int] (w3) at (.5,.3) {};
  \draw (w1) edge (0,-.5)
        (w1) edge (w2) edge (w3) 
        (w2) edge (w3) edge (-.7,-.5)
        (w3) edge (.7,-.5);
 \end{tikzpicture}
\]
If the hairy graph thus produced is not connected, we interpret it as a product (within the symmetric algebra) of its connected components. We call a graph in $\Graphs_n$ \emph{internally connected} if the graph obtained by deleting of all external vertices is non-empty and connected. In particular, the internally connected graphs 
in $\Graphs_{n}(1)$ may be identified with $\HGC_{m,n}$ as a graded vector space.\footnote{Likewise, the internally connected graphs in $\Graphs_n(r)$ can be identified with hairy graphs whose hairs come in $r$ colors. Such graphs are connected to the study of string links \cite{ST}, but we will not pursue them further in this paper.}

It is clear from the combinatorial description \eqref{equ:graphsdiff} of the differential on $\Graphs_n$ that the internally connected graphs form a subcomplex of $\Graphs_n(r)$ for each $r$.
In particular, the differential of $\Graphs_n(1)$ defines a differential on the internally connected subcomplex.
Identifying this subcomplex with $\HGC_{0,n}$ we obtain our desired deformed differential $D$ as the one induced by the differential on $\Graphs_n(1)$.
This differential has the following explicit combinatorial form:
\[
D
\begin{tikzpicture}[baseline=-.65ex]
\node (v) {$\Gamma$};
\draw (v) edge +(-.5,-.5) edge +(0,-.5) edge +(.5, -.5) edge +(-.5,.5) edge +(0,.5) edge +(.5, .5);
\end{tikzpicture}
 =
 \delta \Gamma +
 \sum_S
 \begin{tikzpicture}[baseline=-.65ex]
 \node (v) at (0,.5) {$\Gamma$};
 \node[int] (w) at (0,-.5) {};
\draw (w) edge +(0,-.5)
          (v) edge[bend left] node[right] {$S$} (w) edge[bend right] (w) edge(w) edge +(-.5,.5) edge +(0,.5) edge +(.5, .5);
\end{tikzpicture}\, .
\]
Here $\delta$ is the original (undeformed) differential and the sum on the right-hand side is over all subsets $S$ of the set of hairs with at least two elements.

It is clear that our new differential $D$ on $\HGC_{0,n}$ is indeed a deformation of the original differential $\delta$: Filtering $\HGC_{0,n}$ by the number of hairs, the differential induced by $D$ on the associated graded is exactly $\delta$.
Furthermore, the filtration by the number of hairs gives rise to the spectral sequence of Theorem \ref{thm:main}. 
To show Theorem \ref{thm:main} it hence suffices to show the following result.

\begin{prop}\label{prop:mevenspecs}
The complex $(\HGC_{0,n}, D)$ is acyclic, and the spectral sequence associated to the filtration by arity of the external vertex converges to 0.
\end{prop}
\begin{proof}
The complex $\Graphs_n(1)$, and the internally connected piece splits into a direct sum of finite dimensional subcomplexes according to the Euler characteristic of graphs. Hence the above spectral sequence clearly converges to the cohomology $H(\HGC_{0,n},D)$.
 Hence it suffices to show that this cohomology vanishes.
A proof of this vanishing result may be found in \cite{severa} (see in particular Appendix B therein), where the internally connected subcomplex was denoted by $\mathsf{CG}$. (In fact, loc. cit. only considers the case $n=2$, but this does not play a role for the proof.)


\end{proof}

\subsection{Remark: The image of the ordinary graph cohomology}\label{sec:image of bald}
As noted in (e.g.) \cite{TW2} the ordinary (non-hairy) graph complex $\GC_n$ may be embedded into the hairy complexes $\HGC_{m,n}$ by adding one hair, in all possible ways.
\begin{equation}\label{equ:primedef}
\gamma\mapsto F(\gamma):=
\sum_v 
 \begin{tikzpicture}[baseline=-.65ex]
 \node (v) at (0,.3) {$\gamma$};
\draw (v) edge +(0,-.5);
\end{tikzpicture}
\quad\quad\quad\text{(attach hair at vertex $v$)}
\end{equation}
Furthermore, this map induces an injection on the cohomology level.
In particular, the image of any bald graph cocycle $\gamma\in \GC_n$ yields a cocycle $\Gamma\in \HGC_{0,n}$ which must be a coboundary under the deformed differential $D$ according to Proposition \ref{prop:mevenspecs}.
The purpose of this section is to remark that one has an explicit formula for the element whose $D$-coboundary is $\Gamma$.
To this end let us restrict to the subcomplex $\GC_n^{1vi}\subset \GC_n$ of one-vertex irreducible graphs. (This subcomplex is known to be quasi-isomorphic to the full complex \cite{CV,grt} and hence restricting to the subcomplex does not harm generality.) Define the map 
\begin{equation}\label{equ:onedef}
\begin{gathered}
 \GC_n^{1vi} \to \HGC_{0,n} \\
 \gamma \mapsto \gamma_1 
\end{gathered}
\end{equation}
that sums over all vertices of $\gamma$, deleting the vertex and declaring the incident edges as hairs.

\begin{prop}[]\label{prop:onemapcompat}
The map $\gamma \mapsto \gamma_1$ above satisfies the equation
\[
 (\delta \gamma)_1 = D \gamma_1 \pm F(\gamma).
\]
\end{prop}
In particular, it follows that if $\gamma\in \GC_n$ is a one-vertex irreducible cocycle, then 
\[
 D\gamma_1 =\pm \Gamma.
\]

\subsection{Remark: Compatibility of the differentials}
By Theorem \ref{thm:main} and the results recalled in sections \ref{sec:KWZrecollection} and \ref{sec:TWrecollection} we have several spectral sequences containing the non-hairy and hairy graph cohomology.


We want to remark that all the above deformations of the differential are quite beautifully compatible, 
in the sense that they may be extracted from one deformation of the differential on $\HGC_{0,n}$ ($n=0,1$).
In fact, we saw at the beginning of this section that $\HGC_{0,n}$ may be considered as the internally connected part of $\Graphs_n(1)$, as a graded vector space, and the differential on $\Graphs_n(1)$ yielded the deformed differential $D$ on $\HGC_{0,n}$. But we may as well consider $\HGC_{0,n}$ as the internally connected part of $\Graphs_0^\tp(1)$ for $n=0$ and $\Graphs_0^\Theta(1)$ for $n=1$,\footnote{See the end of section \ref{sec:graphsop} for the definition.} and consider the corresponding induced differentials (say $\tilde D$) on $\HGC_{0,0}$ and $\HGC_{0,1}$. The differential $\tilde D$ deforms $D$ and hence also the original differential $\delta$.

In fact, one has the following results:
\begin{itemize}
  \item One can consider the descending complete filtration on $(\HGC_{0,n},\tilde D)$ by loop number (in $\HGC_n$).\footnote{\label{footn:loops} Note that the number of loops of an element $\Gamma\in \HGC_{0.n}$ differs from the number of loops of the image of that element in $\Graphs_n(1)$ by the number of hairs minus one.} The associated graded complex is the complex $(\HGC_{0,n}, [h_n,\cdot])$ considered in section \ref{sec:TWrecollection}.
 \item One can consider the descending complete filtration on $(\HGC_{0,n},\tilde D)$ by the number of loops in the images of hairy graphs in $\Graphs_{n}(1)$.\textsuperscript{\ref{footn:loops}} The associated graded complex is $(\HGC_{0,n}, D)$ as considered above.
 \item If $n=0$ the map \eqref{equ:onedef} is already a map of complexes $(\GC_{0}, \delta+\nabla)\to (\HGC_{0,n}, \tilde D)$.
 For $n=1$ there is a modified version of the map \eqref{equ:onedef} defined as
 \[
\gamma \mapsto F'(\gamma)
\sum_{k\geq 0}
\frac{1}{(2k+1)!}
\sum
\underbrace{
 \begin{tikzpicture}[baseline=-.65ex]
 \node (v) at (0,.3) {$\Gamma$};
\draw (v) edge +(0,-.5) edge +(-.5,-.5) edge +(.5,-.5);
\end{tikzpicture}
}_{2k+1\times}
\quad\quad\quad\text{(sum over all ways of attaching hairs)}
 \]
 \item Proposition \ref{prop:onemapcompat} extends nicely. Namely, the map \eqref{equ:onedef} satisfies for $n=0$ the compatibility relation 
 \[
  \tilde D \gamma_1 = ((\delta+\nabla)\gamma)_1 \pm F(\gamma)
 \]
and for $n=1$ the relation
\[
  \tilde D \gamma_1 = (\delta_\Theta \gamma)_1 \pm F'(\gamma).
\]
\end{itemize}

\subsection{Picture of the hairy graph cohomology: The waterfall mechanism}\label{sec:waterfalle}
As described in the introduction, the spectral sequences considered above can be used to generate many (non-trivial) hairy graph cohomology classes from non-hairy classes.
Concretely, start with a nontrivial non-hairy graph cocycle $\gamma$. Then the corresponding hairy cocycle $\Gamma=F(\gamma)$ is a non-trivial graph cocycle in $(\HGC_{0,n},\delta)$. Hence it must be killed by some other cocycle $X$ under the deformed differential $D$. In fact, by the remarks of section \ref{sec:image of bald} we may take $X=\gamma_1$.
Now $X$ must either kill or be killed by some other cohomology class $Y$ in the second spectral sequence. This class $Y$ must in turn kill or be killed by another class, etc.

Overall, there is a string of non-trivial hairy graph cohomology classes obtained from any non-hairy class.
The tables in Figure \ref{fig:cancellatione} show the dimensions of the hairy graph cohomology in each bidegree (number of hairs and loop order).
The arrows indicate how classes kill each other in the two spectral sequences of Theorem \ref{thm:main} (blue) and Theorem \ref{thm:TW} (red). (Not all cancellations are shown for the sake of readability.)

\begin{figure}

 \begin{align*}
&  \begin{tikzpicture}
\matrix (mag) [matrix of math nodes,ampersand replacement=\&]
{
{\phantom{2} } \& 1 \& 2 \& 3 \& 4 \& 5 \& 6 \& 7 \& 8 \&  \&  \&  \&  \&  \&   \\
 9 \& 1_{16} \& 1_{16} \&  \&  \&  \&  \&  \&  \&   \&  \& \& \& \& \\
 8 \&       \& 1_{15} \&  \&  \&  \&  \&  \&  \&   \&  \& \& \& \& \\
 7 \& 1_{12} \& 1_{12} \&  \&  \&  \&  \&  \&  \&   \&  \& \& \& \& \\
 6 \&       \& 1_{11} \&  \&  \&  \&  \&  \&  \&  \&  \&  \&  \& \&  \\
 5 \& 1_8    \&  \&  \&   1_{11} \& \&  \&  \&  \&  \&  \&  \&  \& \&  \\
 4 \&       \&  \&  \& 1_8 \&  \&  \&  \&  \&  \&  \&  \&  \& \&  \\
 3 \& 1_4    \&  \&  \& 1_7 \&  \&  \&  \&  \&  \&  \&  \&  \&  \&\\
 2 \&       \&  \&  \&  \&  \&  \&  \&  \&  \&  \&  \&   \& \& \\
 1 \&       \&  \& 1_1 \&  \& 1_1 \& 1_4 \& 1_1 \& 1_1,1_4 \&  \&  \&   \& \& \& \\
};
\draw (mag-1-1.south) -- (mag-1-9.south);
\draw (mag-1-1.east) -- (mag-10-1.east);
\draw[-, purple!30, thick] (mag-3-3.center) edge (mag-2-3.center);
\draw[-, purple!30, thick] (mag-5-3.center) edge (mag-4-3.center);
\draw[-, purple!30, thick] (mag-8-5.center) edge (mag-7-5.center);
\draw[-, blue!30, thick] (mag-10-4) edge (mag-8-2);
\draw[latex-, blue!50, thick] (mag-10-4) edge (mag-8-2);
\draw[-latex, blue!50, thick] (mag-6-2) edge (mag-10-6);
\draw[-latex, blue!50, thick] (mag-8-5) edge (mag-10-7);
\draw[-latex, blue!50, thick] (mag-5-3) edge (mag-7-5);
\draw[-latex, blue!50, thick] (mag-4-3) edge (mag-10-9);
\draw[-latex, blue!50, thick] (mag-6-5) edge (mag-10-9);
\end{tikzpicture}
&\begin{tikzpicture}
\matrix (mag) [matrix of math nodes,ampersand replacement=\&]
{
{\phantom{2} } \& 1 \& 2 \& 3 \& 4 \& 5 \& 6 \& 7 \& 8 \&  \&  \&  \&  \&  \&   \\
 9 \&       \& 1_6 \&  \&  \&  \&  \&  \&  \&   \&  \& \& \& \& \\
 8 \&       \&  \&  \&  \&  \&  \&  \&  \&   \&  \& \& \& \& \\
 7 \& 1_{5} \& 1_5 \& 2_5 \&  \&  \&  \&  \&  \&   \&  \& \& \& \& \\
 6 \&       \&  \& 1_2 \&  \&  \&  \&  \&  \&  \&  \&  \&  \& \&  \\
 5 \&       \& 1_2 \& 2_2 \& 3_2  \& 5_2 \&  \&  \&  \&  \&  \&  \&  \& \&  \\
 4 \&       \&  \& 1_1 \& 1_1 \& 1_{-1} \&  \&  \&  \&  \&  \&  \&  \& \&  \\
 3 \& 1_1   \& 1_1 \& 2_1 \& 2_1 \& 3_1 \&  \&  \&  \&  \&  \&  \&  \&  \&\\
 2 \&       \&  \&  \&  \& 1_{-2} \& 1_{-2} \&  \&  \&  \&  \&  \&   \& \& \\
 1 \&       \&  \& 1_{-2} \& 1_{-2} \& 2_{-2} \& 2_{-2},1_{-5} \& 1_{-5} \&   \&  \&  \&   \& \& \& \\
};
 \draw (mag-1-1.south) -- (mag-1-9.south);
 \draw (mag-1-1.east) -- (mag-10-1.east);
 \draw[-, purple!30, thick] (mag-2-3.center) edge (mag-4-3.center);
 \draw[-, purple!30, thick] (mag-6-3.center) edge (mag-8-3.center);
 \draw[-, purple!30, thick] (mag-5-4.center) edge (mag-7-4.center);
  \draw[-, purple!30, thick] (mag-6-4.center) edge (mag-8-4.center);
  \draw[-, purple!30, thick] (mag-6-5.center) edge (mag-8-5.center);
  \draw[-, purple!30, thick] (mag-6-6.center) edge (mag-8-6.center);
  \draw[-, purple!30, thick] (mag-7-6.center) edge (mag-9-6.center);
 \draw[-latex, blue!30, thick] (mag-8-2) edge (mag-10-4);
 \draw[-latex, blue!30, thick] (mag-8-3) edge (mag-10-5);
 \draw[-latex, blue!30, thick] (mag-8-4) edge (mag-10-6);
 \draw[-latex, blue!30, thick] (mag-8-5) edge (mag-10-7);
 \draw[-latex, blue!30, thick] (mag-6-3) edge (mag-10-7);
  \draw[-latex, blue!30, thick] (mag-6-4) edge (mag-10-8);
 \draw[-latex, blue!30, thick] (mag-4-2) edge (mag-6-4);
 \draw[-latex, blue!30, thick] (mag-4-3) edge (mag-6-5);
\end{tikzpicture}
\end{align*}

  \caption{\label{fig:cancellatione} 
  Computer generated table of the dimensions of the hairy graph cohomology $\text{dim}H(\HGC_{2,2})$ (left) and $\text{dim}H(\HGC_{2,3})$ (right).
  The rows indicate the number of hairs ($\uparrow$), the columns the loop order ($\rightarrow$). A table entry $1_3$ means that there the degree 3 subspace is one-dimensional. 
  The arrows indicate (some of) the cancellations of classes in the two spectral sequences discussed in section \ref{sec:waterfalle}, which we call the ``waterfall mechanism''. The computer program used approximate (floating point) arithmetic, so the displayed numbers should not be considered as rigorous results.
  }
\end{figure}
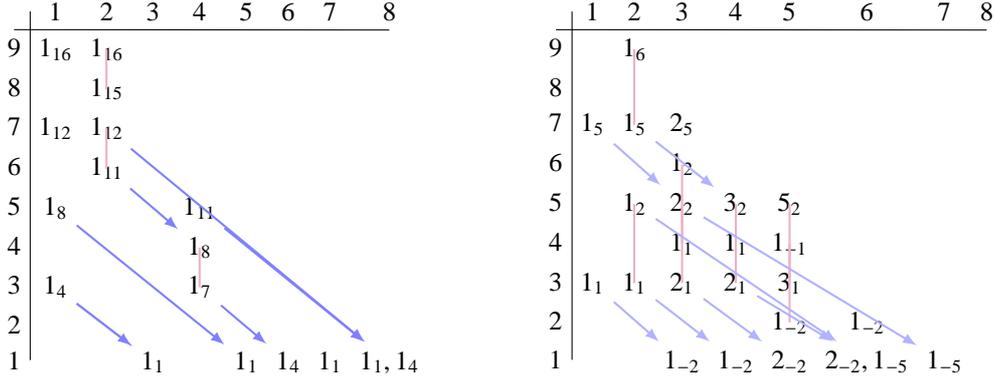

We call this mechanism of creating hairy classes from bald the ``waterfall mechanism''. 
(The cancellation patterns look a bit like streams of water going downhill. The (bald) cohomology in the bottom line is witness that there are streams above it, and a source from which streams originate.)

We note that in the case of even $n$ the waterfall mechanism accounts for all hairy classes in the computer accessible regime.
For odd $n$, we see an additional 4 classes which do not originate from bald classes in hairs/loop order combinations $(1,5)$, $(4,5)$, $(4,3)$ and $(6,3)$.

Note also that the convergence of the spectral sequences to known data is not enough to reconstruct the precise location (i.e., tri-degree) of non-trivial cohomology classes in the ``string'', because we do not know on in general on which pages certain classes are cancelled. However, the picture in Figure \ref{fig:cancellatione} seems to indicate that such cancellations follow a fairly regular pattern.
The study of the abutment properties might hence be an interesting topic of future work.
Indeed, for $n$ even the result of the forthcoming work \cite{TW2} shows that all cancellations of the spectral sequence of section \ref{sec:TWrecollection} happen on the $E^2$-page, in consistency with the numerical tables. 

%
%
%
%

\section{The spectral sequence: $m$ odd}\label{sec:spectralo}
If the source dimension is $m=-1$, i.e., odd, we can define the following additional operation of degree $1$:
\begin{gather*}
\Delta \colon \HGC_{-1,n} \to \HGC_{-1,n} \\
\Gamma \mapsto
\begin{cases}
0 & \quad \text{if $\Gamma$ has exactly one hair.} \\
\sum \pm
\begin{tikzpicture}[baseline=-.65ex]
\node(v) at (0,0) {$\Gamma$};
\draw (v) edge[loop ] (v);
\draw (v) edge +(-.5,-.5) edge +(0,-.5) edge +(.5,-.5) ;
\end{tikzpicture}
\end{cases}
\end{gather*}
where the sum in the second line is over all ways of choosing one of the hairs and reconnecting it to a vertex of $\Gamma$ other than the vertex the hair originated from.

\begin{lemma}
The operation $\Delta$ squares to zero and (anti-)commutes with $\delta$, so that $(\delta+\Delta)^2=0$.
\end{lemma}
\begin{proof}
A straightforward verification.
\end{proof}

Of course, in analogy with Proposition \ref{prop:mevenspecs} we would like to show the following result.
\begin{conjecture}\label{conj:vanishing}\label{conj:main}
$H(\HGC_{-1,n},\delta+\Delta)=0$ for all $n$. 
\end{conjecture}
The conjecture is supported by low loop-order calculations and computer results, cf. Figure \ref{fig:cancellationo}

We may endow $\HGC_{-1,n}$ with a descending complete filtration by loop number. It is clear that the differential $\Delta$ increases the loop number by one, and hence the first page of the associated spectral sequence will agree with $(\HGC_{-1,n}, \Delta)$. 
Let us summarize the situation.

\begin{prop}
Consider the spectral sequence obtained from the filtration by loop order on $(\HGC_{-1,n},\delta+\Delta)$. Its first page is the hairy graph complex $(\HGC_{-1,n},\delta)$. Furthermore, if Conjecture \ref{conj:vanishing} holds, than the spectral sequence converges to 0.
\end{prop}
\begin{proof}
The only non-obvious statement is that the spectral sequence converges to the cohomology. However, it holds in general that the spectral sequence of a descending complete (bounded above) filtration converges to a subspace of the cohomology, which must 0 in this case as the cohomology is (assumed to be) 0.
\end{proof}

\subsection{Tentative picture of the hairy graph cohomology for $m$ odd and the waterfall mechanism}\label{sec:waterfallo}
If Conjecture \ref{conj:vanishing} holds we may again use the two spectral sequences of Theorem \ref{thm:TW} and of the previous section to generate many non-trivial hairy cohomology classes out of bald classes, just as in section \ref{sec:waterfalle}. We will again call the mechanism the ``waterfall mechanism''. The tentative cancellations in the two spectral sequences are illustrated in Figure \ref{fig:cancellationo} in the computer accessible regime.
 
 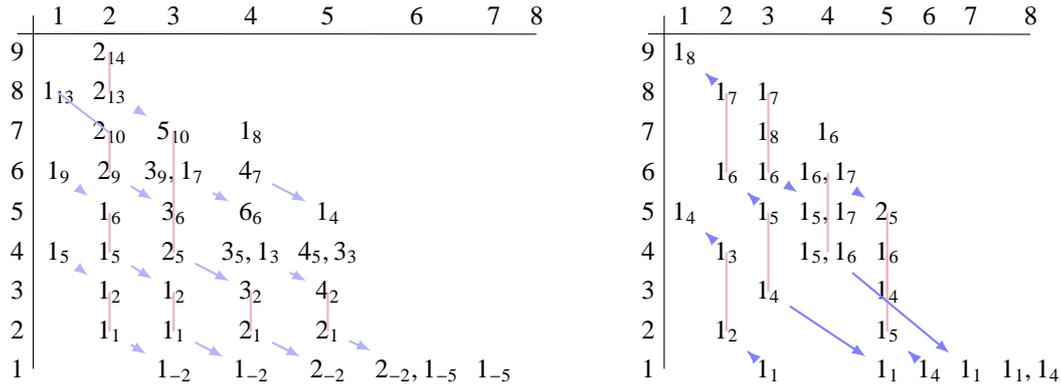
\begin{figure}
\begin{align*}
&
\begin{tikzpicture}
\matrix (mag) [matrix of math nodes,ampersand replacement=\&]
{
{\phantom{2} } \& 1 \& 2 \& 3 \& 4 \& 5 \& 6 \& 7 \& 8 \&  \&  \&  \&  \&  \&   \\
 9 \&       \& 2_{14} \&  \&  \&  \&  \&  \&  \&   \&  \& \& \& \& \\
 8 \&  1_{13}  \&  2_{13} \&  \&  \&  \&  \&  \&  \&   \&  \& \& \& \& \\
 7 \&       \& 2_{10} \& 5_{10} \& 1_8 \&  \&  \&  \&  \&   \&  \& \& \& \& \\
 6 \&  1_9 \& 2_9  \& 3_9,1_7 \& 4_7 \&  \&  \&  \&  \&  \&  \&  \&  \& \&  \\
 5 \&       \& 1_6 \& 3_6 \& 6_6  \& 1_4 \&  \&  \&  \&  \&  \&  \&  \& \&  \\
 4 \&  1_5   \& 1_5 \& 2_5 \& 3_5,1_3 \& 4_{5},3_3 \&  \&  \&  \&  \&  \&  \&  \& \&  \\
 3 \&      \&  1_2 \& 1_2 \& 3_2  \& 4_2 \&  \&  \&  \&  \&  \&  \&  \&   \&\\
 2 \&       \&   1_{1} \& 1_{1} \& 2_{1} \& 2_{1} \&  \&  \&   \&  \& \&  \&   \& \& \\
 1 \&       \&  \& 1_{-2} \& 1_{-2} \& 2_{-2} \& 2_{-2},1_{-5} \& 1_{-5} \&   \&  \&  \&   \& \& \& \\
};
 \draw (mag-1-1.south) -- (mag-1-9.south);
 \draw (mag-1-1.east) -- (mag-10-1.east);
 \draw[-, purple!30, thick] (mag-2-3.center) edge (mag-3-3.center);
 \draw[-, purple!30, thick] (mag-4-3.center) edge (mag-5-3.center);
 \draw[-, purple!30, thick] (mag-6-3.center) edge (mag-7-3.center);
 \draw[-, purple!30, thick] (mag-6-4.center) edge (mag-7-4.center);
  \draw[-, purple!30, thick] (mag-5-4.center) edge (mag-6-4.center);
  \draw[-, purple!30, thick] (mag-4-4.center) edge (mag-5-4.center);
  \draw[-, purple!30, thick] (mag-9-5.center) edge (mag-8-5.center);
  \draw[-, purple!30, thick] (mag-9-6.center) edge (mag-8-6.center);
   \draw[-, purple!30, thick] (mag-9-3.center) edge (mag-8-3.center);
    \draw[-, purple!30, thick] (mag-9-4.center) edge (mag-8-4.center);
 \draw[-latex, blue!30, thick] (mag-9-3) edge (mag-10-4);
 \draw[-latex, blue!30, thick] (mag-9-4) edge (mag-10-5);
 \draw[-latex, blue!30, thick] (mag-9-5) edge (mag-10-6);
 \draw[-latex, blue!30, thick] (mag-9-6) edge (mag-10-7);
  \draw[-latex, blue!30, thick] (mag-7-2) edge (mag-8-3);
 \draw[-latex, blue!30, thick] (mag-7-3) edge (mag-8-4);
 \draw[-latex, blue!30, thick] (mag-7-4) edge (mag-8-5);
 \draw[-latex, blue!30, thick] (mag-7-5) edge (mag-8-6);
 \draw[-latex, blue!30, thick] (mag-5-2) edge (mag-6-3);
  \draw[-latex, blue!30, thick] (mag-5-3) edge (mag-6-4);
 \draw[-latex, blue!30, thick] (mag-5-4) edge (mag-6-5);
 \draw[-latex, blue!30, thick] (mag-5-5) edge (mag-6-6);
 \draw[-, blue!30, thick] (mag-3-2.center) edge (mag-4-3.center);
    \draw[-latex, blue!30, thick] (mag-3-3) edge (mag-4-4);
\end{tikzpicture}
&
\begin{tikzpicture}
\matrix (mag) [matrix of math nodes,ampersand replacement=\&]
{
{\phantom{2} } \& 1 \& 2 \& 3 \& 4 \& 5 \& 6 \& 7 \& 8 \&  \&  \&  \&  \&  \&   \\
 9 \&  1_8 \&  \&  \&  \&  \&  \&  \&  \&   \&  \& \& \& \& \\
 8 \&       \& 1_7 \& 1_7 \&  \&  \&  \&  \&  \&   \&  \& \& \& \& \\
 7 \&       \&  \& 1_8  \&  1_6 \&  \&  \&  \&  \&   \&  \& \& \& \& \\
 6 \&       \& 1_6 \& 1_6 \& 1_6,1_7 \&  \&  \&  \&  \&  \&  \&  \&  \& \&  \\
 5 \&  1_4 \&  \& 1_5 \&  1_5,1_7  \& 2_5 \&  \&  \&  \&  \&  \&  \&  \& \&  \\
 4 \&       \& 1_3 \&  \&  1_5,1_6 \& 1_6 \&  \&  \&  \&  \&  \&  \&  \& \&  \\
 3 \&       \&  \& 1_4 \&  \& 1_4 \&  \&  \&  \&  \&  \&  \&  \&  \&\\
 2 \&       \& 1_2 \&  \&  \& 1_5  \&  \&  \&  \&  \&  \&  \&   \& \& \\
 1 \&       \&  \& 1_1 \&  \& 1_1 \& 1_4 \& 1_1 \& 1_1,1_4 \&  \&  \&   \& \& \& \\
};
\draw (mag-1-1.south) -- (mag-1-9.south);
\draw (mag-1-1.east) -- (mag-10-1.east);
\draw[-, purple!30, thick] (mag-3-3.center) edge (mag-5-3.center);
\draw[-, purple!30, thick] (mag-7-3.center) edge (mag-9-3.center);
\draw[-, purple!30, thick] (mag-6-4.center) edge (mag-8-4.center);
\draw[-, purple!30, thick] (mag-5-5.center) edge (mag-7-5.center);
\draw[-, purple!30, thick] (mag-3-4.center) edge (mag-5-4.center);
\draw[-, purple!30, thick] (mag-7-6.center) edge (mag-9-6.center);
\draw[-, purple!30, thick] (mag-6-6.center) edge (mag-8-6.center);
\draw[-latex, blue!50, thick] (mag-9-3) edge (mag-10-4);
\draw[-latex, blue!50, thick] (mag-8-4) edge (mag-10-6);
\draw[-latex, blue!50, thick] (mag-6-2) edge (mag-7-3);
\draw[-latex, blue!50, thick] (mag-5-3) edge (mag-6-4);
\draw[-latex, blue!50, thick] (mag-2-2) edge (mag-3-3);
\draw[-latex, blue!50, thick] (mag-5-4) edge (mag-6-5);
\draw[-latex, blue!50, thick] (mag-5-5) edge (mag-6-6);
\draw[-latex, blue!50, thick] (mag-9-6) edge (mag-10-7);
\draw[-latex, blue!50, thick] (mag-7-5) edge (mag-10-8);
\end{tikzpicture}
 \end{align*}
 \caption{\label{fig:cancellationo} 
  Computer generated table of the dimensions of the hairy graph cohomology $\text{dim}H(\HGC_{1,2})$ (left) and $\text{dim}H(\HGC_{3,3})$ (right).
  The rows indicate the number of hairs ($\uparrow$), the columns the genus ($\rightarrow$). A table entry $1_3$ means that there the degree 3 subspace is one-dimensional. 
  The arrows indicate (some of) the tentative cancellations of classes in the two spectral sequences discussed in section \ref{sec:waterfallo}. The computer program used approximate (floating point) arithmetic, so the displayed numbers should not be considered as rigorous results.
  }
 \end{figure}

\end{document}